\documentclass[12pt]{amsart}

\usepackage[T1]{fontenc}
\usepackage{amsmath, amssymb}
\usepackage[hidelinks]{hyperref}
\usepackage[english]{babel}
\usepackage{mathrsfs}
\usepackage{eucal}
\usepackage[all]{xy}
\usepackage{tikz-cd}

\newtheorem{thm}{Theorem}[section]

\newtheorem*{thm*}{Theorem}
\newtheorem{lem}[thm]{Lemma}
\newtheorem{fact}[thm]{Fact}

\newtheorem{prop}[thm]{Proposition}
\newtheorem*{prop*}{Proposition}

\newtheorem{cor}[thm]{Corollary}
\newtheorem*{cor*}{Corollary}

\theoremstyle{definition}
\newtheorem{defn}[thm]{Definition}
\newtheorem*{defn*}{Definition}

\newtheorem*{question*}{Question}

\newtheorem*{conv*}{Convention}

\def\bb{\mathbb}

\def\bb{\mathbb}

\def\cal{\mathcal}

\makeatletter

\def\dotminussym#1#2{%
  \setbox0=\hbox{$\m@th#1-$}%
  \kern.5\wd0%
  \hbox to 0pt{\hss\hbox{$\m@th#1-$}\hss}%
  \raise.6\ht0\hbox to 0pt{\hss$\m@th#1.$\hss}%
  \kern.5\wd0}

\def \sl{\operatorname{SL}}
\def \psl{\operatorname{PSL}}

\title{A non-uniformly inner amenable group}
\author{Isaac Goldbring}

\address{Department of Mathematics\\University of California, Irvine, 340 Rowland Hall (Bldg.\# 400),
Irvine, CA 92697-3875}
\email{isaac@math.uci.edu}
\urladdr{http://www.math.uci.edu/~isaac}
 \thanks{The author was partially supported by NSF grant DMS-2054477.}

\begin{document}

\begin{abstract}
  We provide an example of two elementarily equivalent countable ICC groups $G$ and $H$ such that $G$ is amenable and $H$ is not inner amenable.  As a result, we provide the first example of elementarily equivalent groups whose group von Neumann algebras are not elementarily equivalent, answering a question asked by many researchers.
\end{abstract}
\maketitle

\section{Introduction}

If $G$ is a group and $X$ is a set, an action $G\curvearrowright X$ of $G$ on $X$ is called \textbf{amenable} if there is a finitely additive $G$-invariant probability measure on $X$.  $G$ is called \textbf{amenable} if the action of $G$ on itself by left multiplication is amenable while $G$ is called \textbf{inner amenable} if the action of $G$ on $G\setminus \{e\}$ by conjugation is amenable.  Throughout this note, we assume basic facts about amenable groups; we recommend \cite{juschenko} as a good reference.  We note that amenable groups are inner amenable; see, for example, \cite[Theorem 2.20]{juschenko}.  

A group $G$ is called \textbf{ICC} if all nontrivial conjugacy classes are infinite. 

Two groups $G$ and $H$ are called \textbf{elementarily equivalent} if they satisfy the same first-order sentences in the language of groups.  Equivalently, by the Keisler-Shelah theorem, two groups are elementarily equivalent if and only if they have isomorphic ultrapowers.

The main result of this note is the following:

\begin{thm}\label{main}
There are elementarily equivalent countable ICC groups $G$ and $H$ such that $G$ is amenable and $H$ is not inner amenable.
\end{thm}

In Section 3, we show how our main result settles a question in the model theory of operator algebras while in Section 4 we explain how to interpret the main result as the existence of an inner amenable but not uniformly inner amenable group.

We thank David Jekel, Yash Lodha, and Jennifer Pi for many useful discussions around this project.

\section{Proof of the main result}

Throughout this note, $K:=\bb F_2^{alg}$ denotes the algebraic closure of the field of two elements while $L$ denotes any algebraically closed field of characteristic $2$ and positive transcendence degree (for example, $L=\bb F_2(t)^{alg}$, where $t$ is an indeterminate).

Recall that for any field $F$, $\sl_2(F)$ denotes the group of $2\times 2$ matrices over $F$ of determinant $1$.

$\sl_2(K)$, being a locally finite group, is amenable.  In fact, for any field $F$, $\sl_2(F)$ is amenable if and only if $F$ is algebraic over a finite field (see, for example, \cite[Proposition 11]{bekka}).  Consequently, we have that 
$\sl_2(L)$ is \emph{not} amenable.


We actually claim that $\sl_2(L)$ is not even \emph{inner} amenable.  To see this, we will need the following fact (for a proof, see, for example, \cite[Theorem 2.21]{juschenko}):

\begin{fact}
Suppose that $G\curvearrowright X$ is an amenable action and that the stabilizer subgroup $G_x$ is amenable for each $x\in X$.  Then $G$ is amenable.
\end{fact}

By considering the conjugation action of $G$ on $G\setminus \{e\}$, one has, as a corollary, the following:

\begin{fact}\label{mainfact}
If $G$ is inner amenable and the centralizer $C(g)$ of each $g\in G\setminus \{e\}$ is amenable, then $G$ is amenable.
\end{fact}

Recall that a group $G$ is called \textbf{commutative transitive (CT)} if  $C(g)$ is abelian for all $g\in G\setminus \{e\}$.  (The nomenclature is due to the fact that CT groups are those groups for which, on nontrivial elements, commutation is a transitive relation.)  Since abelian groups are amenable, the previous fact implies:

\begin{cor}\label{innerCT}
An inner amenable CT group is amenable. 
\end{cor} 

Thus, to prove that $\sl_2(L)$ is not inner amenable, it suffices to quote the following:

\begin{fact}\label{slipperyfact}
For any field $F$ of characteristic $2$, $\sl_2(F)$ is a CT group.  
\end{fact}

Fact \ref{slipperyfact} has a slippery history.  In \cite{weisner}, it was mentioned that $\psl_2(\mathbb F_{2^n})$ is a CT group for any $n\geq 1$.  (Recall that for any field $F$, $\psl_2(F):=\sl_2(F)/\{\pm 1\}$; when $F$ has characteristic $2$, $\psl_2(F)$ coincides with $\sl_2(F)$.)  Strangely enough, no proof of this fact is given, but the author claims ``That this group possesses the property in question follows from its analysis.''  Then a reference is given to \cite[pages 262-265]{dickson}.  From this fact and some basic model theory, one can prove Fact \ref{slipperyfact}; see \cite[Theorem 3.7(1)]{CT}).  Given that this ``proof'' of Fact \ref{slipperyfact} is somewhat incomplete, we prefer to give a detailed proof here; we thank David Jekel for providing us with this proof and for his permission to include it here.  (Not that we will need to know this but, building on work of Suzuki \cite{suzuki} and Wu \cite{Wu}, it follows that $\psl_2(F)$ is \emph{not} a CT group when $F$ has positive characteristic different from $2$; see \cite[Theorem 3.7(2)]{CT}.)  

\begin{proof}[Proof of Fact \ref{slipperyfact}]
Without loss of generality, we may assume that $F$ is algebraically closed.  Fix $g\in \sl_2(F)\setminus \{1\}$.  Since $F$ is algebraically closed, $g$ is similar to a matrix $g'$ in $\sl_2(F)$ that is in Jordan canonical form (see \cite[Chapter XIV, Corollary 2.5]{lang}).  Since $g$ and $g'$ have isomorphic centralizers in $\sl_2(F)$, we may assume that $g$ itself is in Jordan canonical form.

First suppose that $g$ is diagonal.  Then $g=\begin{pmatrix} a & 0 \\ 0 & a^{-1}\end{pmatrix}$ for some $a\in F$.  Fix $h=\begin{pmatrix}
c & d \\ e & f
\end{pmatrix}\in C(g)$.  Since $gh=\begin{pmatrix}
    ac & ad \\ a^{-1}e & a^{-1} f
\end{pmatrix}$ while $hg=\begin{pmatrix}
    ac & a^{-1}d \\ ae & a^{-1}f
\end{pmatrix}$, we see that $(a-a^{-1})d=(a-a^{-1})e=0$; since $a\not=1$ (else $g=1$), we have that $a-a^{-1}\not=0$ (recall that $F$ has characteristic $2$), whence $d=e=0$.  It follows that $C(g)$ is the set of diagonal matrices in $\sl_2(F)$, which is certainly abelian.

Now suppose that $g$ is not diagonal.  In this case, we have that $g=\begin{pmatrix}
    1 & 1 \\ 0 & 1
\end{pmatrix}$ for some $a\in F\setminus\{0\}$.  Once again, fix $h=\begin{pmatrix}
c & d \\ e & f
\end{pmatrix}\in C(g)$.  This time we have $gh=\begin{pmatrix}
    c+e & d+f \\ e & f
\end{pmatrix}$ while $hg=\begin{pmatrix}
    c & c+d \\ e & e+f
\end{pmatrix}$.  We thus see that $e=0$ and $c=f$.  It follows that $C(g)$ is the set of upper triangular matrices in $\sl_2(F)$ whose diagonal elements are the same; this group is easily checked to be abelian.
\end{proof}

Recalling that $\sl_2(L)$ is not amenable, Corollary \ref{innerCT} and Fact \ref{slipperyfact} imply that:

\begin{cor}
$\sl_2(L)$ is not inner amenable.
\end{cor}

Finally, we record:

\begin{lem}\label{eelemma}
$\sl_2(K)$ and $\sl_2(L)$ are elementarily equivalent.
\end{lem}

\begin{proof}
Since $K$ and $L$ are algebraically closed fields of the same characteristic, we have that $K$ and $L$ are elementarily equivalent (see \cite[Proposition 2.2.5]{marker}).  It remains to quote a standard fact: if two fields $F_1$ and $F_2$ are elementarily equivalent, then so are the groups $\sl_2(F_1)$ and $\sl_2(F_2)$.  (There are many ways to see this, but perhaps the quickest is to use the Keisler-Shelah theorem to conclude that $K$ and $L$ have isomorphic ultrapowers and then use the fact that taking $\sl_2$ of a field commutes with ultrapowers.)
\end{proof} 

We note the following well-known fact.  

\begin{prop}
For any algebraically closed field $F$ of characteristic $2$, $\sl_2(F)$ is an ICC group.
\end{prop}

\begin{proof}
Fix $g\in \sl_2(F)\setminus \{e\}$; as before, we may suppose that $g$ is in Jordan canonical form.  If $g=\begin{pmatrix}
    a & 0 \\ 0 & a^{-1}
\end{pmatrix}$, then it remains to note that
$$\begin{pmatrix}
    b & 1 \\ 0 & b^{-1}
\end{pmatrix}^{-1} \begin{pmatrix}
    a & 0 \\ 0 & a^{-1}
\end{pmatrix} \begin{pmatrix}
    b & 1 \\ 0 & b^{-1}
\end{pmatrix}=\begin{pmatrix}
    a & (a+a^{-1})b \\ 0 & a^{-1}
\end{pmatrix}.$$  Since $a+a^{-1}\not=0$ (lest $g=e$, recalling that the characteristic of $F$ is $2$), we see that by letting $b$ vary over $F\setminus \{0\}$, we find infinitely many conjugates of $g$.  

The other case is that $g=\begin{pmatrix}
    1 & 1 \\ 0 & 1
\end{pmatrix}$ with $a\in F\setminus\{0\}$.  In this case, we note that, for $b,c\in F$ with $bc=1$, we have 

$$\begin{pmatrix}
    b & 0 \\ 0 & c
\end{pmatrix}^{-1} \begin{pmatrix}
    1 & 1 \\ 0 & 1
\end{pmatrix} \begin{pmatrix}
    b & 0 \\ 0 & c
\end{pmatrix}=\begin{pmatrix}
    1 & c^2 \\ 0 & 1
\end{pmatrix}.$$
Once again, letting $c$ range over $F\setminus \{0\}$, we find infinitely many conjugates of $g$.
\end{proof}

The previous proposition holds in greater generality:  $\psl_2(F)$ is ICC for any algebraically closed field of positive characteristic; we assumed that the field had characteristic $2$ just to give a simpler proof and since this is the only case that we need.

\section{An application to operator algebras}

Theorem \ref{main} above allows us to provide a negative solution to a question asked by many researchers in the model theory of operator algebras, namely, does elementary equivalence of groups imply the elementary equivalence of their group von Neumann algebras?

\begin{prop}
There are elementarily equivalent countable ICC groups $G$ and $H$ such that their group von Neumann algebras $L(G)$ and $L(H)$ are not elementarily equivalent.
\end{prop}

\begin{proof}
The groups $G:=\sl_2(K)$ and $H:=\sl_2(L)$ yield the desired example.  Indeed, since $G$ is a countable, amenable, ICC group, by Connes' landmark result in \cite{connes}, $L(G)$ is the hyperfinite II$_1$ factor $\cal R$.  However, by a result of Effros \cite{effros}, since $H$ is not inner amenable, $L(H)$ does not have property Gamma.  Since $\cal R$ has property Gamma and property Gamma is an axiomatizable property of tracial von Neumann algebras (see \cite[3.2.2]{mtoa3}), we have that $L(G)$ and $L(H)$ are not elementarily equivalent.
\end{proof}

\section{Uniformly inner-amenable groups}

If $G$ is a group, $S\subseteq G$ is finite, and $\epsilon>0$, then a \textbf{$(S,\epsilon)$-F\o lner set in $G$} is a finite set $T\subseteq G$ such that $|gT\triangle T|<\epsilon |T|$ for all $g\in S$.  A group $G$ is amenable if and only if $(S,\epsilon)$-F\o lner sets in $G$ exist for all finite $S\subseteq G$ and all $\epsilon>0$.  $G$ is said to be \textbf{uniformly amenable} if there is a function $f:\bb N\to \bb N$ such that, for all finite $S\subseteq G$, if $|S|\leq n$, then there is a $(S,\frac{1}{n})$-F\o lner set $T\subseteq G$ with $|T|\leq f(n)$.
In \cite{keller}, Keller studied uniformly amenable groups, showing that $G$ is uniformly amenable if and only if some (equivalently any) nonprincipal ultrapower of $G$ is amenable.  (Keller used the language of nonstandard extensions rather than ultrapowers, but the arguments are identical in either case.)  It then follows that $G$ is uniformly amenable if and only if $H$ is amenable whenever $H$ is elementarily equivalent to $G$.  Indeed, one direction follows from the fact that an ultrapower of $G$ is elementarily equivalent to $G$, while the other direction follows from the fact that any group elementarily equivalent to $G$ embeds (elementarily) into some ultrapower of $G$ and the fact that subgroups of amenable groups are amenable.  

There are plenty of uniformly amenable groups; for example, any solvable group is uniformly amenable.  An example of an amenable group that is not uniformly amenable is the locally finite group $S_\infty:=\bigcup_{n\geq 1}S_n$ as it is fairly straightforward to find a copy of a nonabelian free group inside of an ultrapower of $S_\infty$.

In \cite{KTD2021dynamical}, an appropriate notion of F\o lner set for inner amenable groups was discussed (but not named): 

\begin{defn}
For nonempty finite subsets $S,T\subseteq G$ and $\epsilon>0$,  we say that $T$ is a \textbf{$(S,\epsilon)$-c-F\o lner set} if $|gTg^{-1}\triangle T|<\epsilon|T|$ for all $g\in S$.
\end{defn}

The following fact is mentioned in \cite[Section 6]{KTD2021dynamical}:

\begin{fact}
A group $G$ is inner amenable if and only if:  for all finite $S\subseteq G$ and all $\epsilon>0$, there are arbitrarily large  $(S,\epsilon)$-c-F\o lner sets in $G$.
\end{fact}

We say that $G$ is \textbf{uniformly inner amenable} if there is a function $f:\bb N\to \bb N$ such that, for all finite subsets $S\subseteq G$ with $|F|\leq n$, there is a finite subset $T\subseteq G$ with $n\leq |T|\leq f(n)$ that is a $(S,\frac{1}{n})$-c-F\o lner set for $G$.  One can prove an analogue of Keller's result for amenable groups as follows:

\begin{prop}\label{uniformlyamenableprop}
For a group $G$, the following are equivalent:
\begin{enumerate}
    \item $G$ is uniformly inner amenable.
    \item If $H$ is elementarily equivalent to $G$, then $H$ is inner amenable.
    \item Every ultrapower of $G$ is inner amenable.
\end{enumerate}
\end{prop}

Before beginning the proof of Proposition \ref{uniformlyamenableprop}, we record one fact.  Recall that if $H$ is a subgroup of $G$, then $H$ is \textbf{existentially closed in $G$} if:  for every existential sentence $\varphi$ with parameters in $H$, if $\varphi$ is true in $G$, then $\varphi$ is true in $H$.  Note that elementary subgroups are in particular existential subgroups.  Recall also that, unlike the case of amenable groups, subgroups of inner amenable groups need not be inner amenable.  For example, $G\times \bb Z$ is inner amenable for \emph{any} group $G$.

\begin{lem}\label{existentiallemma}
If $G$ is an inner amenable group and $H$ is a subgroup of $G$ that is existentially closed in $G$, then $H$ is inner amenable.
\end{lem}

\begin{proof}
Fix a finite set $S\subseteq H$, $\epsilon>0$, and $n\geq 1$.  Since $G$ is inner amenable, there is a finite set $T\subseteq G$ that is a $(S,\epsilon)$-c-F\o lner set with $|T|\geq n$.  The existence of this $(S,\epsilon)$-c-F\o lner set fact can be expressed by an existential sentence with parameters from $H$ and is thus true in $H$ since $H$ is existentially closed in $G$.
\end{proof}

\begin{proof}[Proof of Proposition \ref{uniformlyamenableprop}]
(1) $\Rightarrow (2)$:  Suppose that $G$ is uniformly inner amenable as witnessed by a function $f$.  Then for any given $n\geq 1$, the fact that any finite subset of $G$ of size at most $n$ has a c-F\o lner set with error $\frac{1}{n}$ of size in between $n$ and $f(n)$ can be expressed by a first-order sentence true of $G$ and thus of any group elementarily equivalent to $G$.  

(2) $\Rightarrow (1)$:  Suppose that $G$ is not uniformly inner amenable.  Without loss of generality, we may suppose that $G$ is inner amenable.  Then there is some $n\geq 1$ such that, for all $m\geq 1$, there exist finite subsets $S_m\subseteq G$ with $|S_m|\leq n$ such that all $(S_m,\frac{1}{n})$-c-F\o lner sets in $G$ of size at least $n$ have cardinality at least $m$.  By the compactness theorem, there is an elementary extension $H$ of $G$ which contains a subset $S\subseteq H$ with $|S|\leq n$ for which there is no $(S_m,\frac{1}{n})$-c-F\o lner set in $H$ of cardinality at least $n$.  Then $H$ is not inner amenable.

(2) $\Rightarrow$ (3) follows from the fact that ultrapowers of $G$ are elementarily equivalent to $G$. 

(3) implies (2):  Suppose that $H$ is elementarily equivalent to $G$ and embed $H$ elementarily in an ultrapower $G^\mathcal{U}$ of $G$.  By (3), $G^\mathcal{U}$ is inner amenable whence so is $H$ by Lemma \ref{existentiallemma}.
\end{proof}

Theorem \ref{main} can thus be stated:

\begin{cor}
$\sl_2(K)$ is an amenable group that is not uniformly inner amenable.
\end{cor}

In an earlier version of this paper, we asked if there is an inner amenable, \emph{nonamenable} group that is not uniformly inner amenable.  In response to this question, Jesse Peterson informed us of the following fact:

\begin{fact}
If $G$ and $H$ are groups, then $G\times H$ is inner amenable if and only if at least one of $G$ or $H$ is inner amenable. 
\end{fact}

The previous fact appears to be folklore but we were unable to find it explicitly stated in the literature.  One direction is easy:  if $G$ is inner amenable and $\mu$ is a finitely additive measure on $G\setminus \{e_G\}$ witnessing that $G$ is inner amenable, then $\mu\times \delta_{e_H}$ is a measure on $(G\times H)\setminus \{(e_G,e_H)\}$ that witnesses that $G\times H$ is inner amenable.  For a proof of the other direction, see \cite[Proposition 2.4]{graphinner}.

As a result of this fact, we see that, for any non-inner amenable group $G$, we have that $G\times \sl_2(K)$ is inner amenable, nonamenable, and not uniformly inner amenable.  To see that $G\times \sl_2(K)$ is not uniformly inner amenable, note that $(G\times \sl_2(K))^{\cal U}\cong G^{\cal U} \times \sl_2(K)^{\cal U}$, which is not inner amenable as neither $G^{\cal U}$ nor $\sl_2(K)^{\cal U}$ are inner amenable ($G^{\cal U}$ is not inner amenable by Lemma \ref{existentiallemma}).


\end{document}